\documentclass[11pt,reqno]{amsart}
\usepackage{amssymb}
\usepackage{amsmath}
\usepackage[T1]{fontenc}
\usepackage[utf8]{inputenc}
\usepackage{mathtools}
\usepackage{enumitem}
\usepackage{anysize}
\usepackage[colorlinks]{hyperref}
\marginsize{2.5cm}{2.5cm}{2.5cm}{2.5cm}
\newcommand\numberthis{\addtocounter{equation}{1}\tag{\theequation}}

\newcommand{\Z}{\mathbb{Z}}

\newcommand{\Q}{\mathbb{Q}}
\newcommand{\QQ}{\mathbb{Q}}

\newcommand{\mmod}{\hspace{-5pt}\mod}

\theoremstyle{plain}
\newtheorem{theorem}{Theorem}
\newtheorem{proposition}[theorem]{Proposition}

\theoremstyle{remark}
\newtheorem{remark}[theorem]{Remark}
\newtheorem{conjecture}[theorem]{Conjecture}

\title{Rational $D(q)$-quintuples}

\author{Goran Dra\v zi\' c}
\address{Faculty of Food Technology and Biotechnology, University of Zagreb, Pierottijeva 6, 10000 Zagreb, Croatia}
\email{gdrazic@pbf.hr}


\begin{document}

\begin{abstract}

For a nonzero rational number $q$, a \emph{rational $D(q)$-$n$-tuple} is a set of $n$ distinct nonzero rationals $\{a_1, a_2, \dots, a_n\}$ such that $a_ia_j+q$ is a square for all $1 \leqslant i < j \leqslant n$. We investigate for which $q$ there exist
infinitely many rational $D(q)$-quintuples. We show that assuming the Parity Conjecture for the twists
of several explicitly given elliptic curves, the density of such $q$ is at least $295026/296010\approx 99.5\%$.

\end{abstract}

\maketitle
\section{Introduction}

Let $q\in \QQ$ be a nonzero rational number. A set of $n$ distinct nonzero rationals $\lbrace a_1, a_2, \dots, a_n\rbrace$ is called a rational $D(q)$-$n$-tuple if $a_ia_j+q$ is a square for all $1 \leqslant i<j\leqslant n.$ If $\lbrace a_1,a_2,\dots,a_n\rbrace$ is a rational $D(q)$-$n$-tuple, then for all nonzero $r\in \QQ,\lbrace ra_1,ra_2,\dots,ra_n\rbrace$ is a $D(qr^2)$-$n$-tuple, since $(ra_1)(ra_2)+qr^2=(a_1a_2+q)r^2$. With this in mind, we restrict to square-free integers $q.$
For a historical overview of Diophantine $m$-tuples and rational $D(q)$-$m$-tuples, we refer the reader to \cite{DujeAMS}, \cite[Sections 14.6 and 16.7.]{DujeTB}, as well as the webpage of Andrej Dujella. \footnote{https://web.math.pmf.unizg.hr/\textasciitilde duje/dtuples.html}

The goal of this paper is to find squarefree integers $q$ for which there exist infinitely many rational $D(q)$-quintuples. In \cite{ANote}, Dujella proved there exist infinitely many rational $D(q)$-quadruples for every rational $q,$ and in \cite{DraKa}, Dražić and Kazalicki, for given $q\in \QQ,$ parametrized all $m\in \QQ$ such that there exists a rational $D(q)$-quadruple $(a,b,c,d)$ with $abcd=m.$ Dujella and Fuchs in \cite{RatQuint} proved that, assuming the Parity Conjecture for the twists of an explicitly given elliptic curve (isomorphic to $E^{(7)},$ details in Tables \ref{table:1},\ref{table:2}), the density of $q\in \QQ$ such that there exist infinitely many rational $D(q)$-quintuples is at least $1/2$. In this paper, also assuming the Parity Conjecture for twists of explicit elliptic curves, we improve the density bound to at least $295026/296010\approx 99.5\%$. 

In \cite{Euler}, Dujella constructed rational $D(q)$-quintuples the form $\lbrace A, B, C, D, x^2 \rbrace$, with $q=\alpha x^2.$ In Section \ref{sec:02} we expand his construction. In Section \ref{sec:03}, we define the curve $\mathcal{C} / \Q(u)$ by
\begin{equation}\label{eq01}
\mathcal{C} \colon \quad z_1^2=f_4(u)c^4+f_3(u)c^3+f_2(u)c^2+f_1(u)c+f_0(u),
\end{equation}
where $f_i(u)$ are rational functions in $\QQ(u)$ explicitly stated at (\ref{eq22}).

The curve $\mathcal{C}$ has a rational point when $c=1,$ so it is birationally equivalent to an elliptic curve $E/\QQ(u).$ The Mordell-Weil group $E(\QQ(u))$ has rank at least five, as we found five independent rational points, which we list at (\ref{eq23}).

Let $q(u)$ be a rational function in variable $u,$ not identically zero. We call a set of $n$ distinct, not identically zero rational functions $\lbrace a_1(u), a_2(u), \dots, a_n(u)\rbrace$ a \textit{$D(q(u))$-$n$-tuple with elements in $\Q(u),$} if $a_i(u)a_j(u)+q(u)=h^2_{i,j}(u), \: h_{i,j} \in \QQ(u)$ for all $1\leq i < j \leq n.$ We will refer to such quintuples more briefly as $D(q(u)$-quintuples.

Every rational point on $E$ determines a $D\left(\alpha(u) x(u)^2\right)$-quintuple $\lbrace A(u),B(u),C(u),D(u),x^2(u)\rbrace,$ provided that no two elements of the quintuple are equal and that no element nor $\alpha(u)$ is identically zero. This connection is \hyperlink{EnaPetorku}{explained} in Section \ref{sec:03}

Fix a squarefree $q\in \mathbb{Z}$ and assume for a moment that $\alpha(u_1) x(u_1)^2=qs_1^2$ for some rationals $u_1,s_1$ such that $s_1\neq 0.$ Then $\lbrace A(u_1)/s_1, B(u_1)/s_1, C(u_1)/s_1, D(u_1)/s_1, x^2(u_1)/s_1 \rbrace$ is a rational $D\left(\alpha(u_1)\cdot \left(x(u_1)/s_1\right)^2\right)$- quintuple, that is, a rational $D(q)$-quintuple. The following reasoning was used by Dujella and Fuchs in \cite{RatQuint}: If we find infinitely many rationals $(u_1,s_1)$ such that
\begin{equation}\label{eq02}
\alpha(u_1)=q\left(\frac{s_1}{x(u_1)}\right)^2
\end{equation} 
then there are infinitely many rational $D(q)$-quintuples.

Let $P(u)$ be the squarefree polynomial such that 
\begin{equation}\label{eq03}
 P(u)\equiv \alpha(u) x(u)^2 \mod (\QQ(u)^{\ast})^2. 
\end{equation}
$P(u)$ is uniquely determined up to scaling by a rational square. Solving (\ref{eq02}) is the same as finding a rational solution $(u_1,s_1)$ of
\begin{equation}\label{eq04}
P(u)=qs^2.
\end{equation} 
If $\deg(P(u))\geq 5,$ then equation (\ref{eq04}) defines a curve of genus at least two, which by Faltings' theorem has only finitely many solutions.
Thus, only if $\deg(P(u))\in \lbrace 1, 2, 3, 4\rbrace$ can we hope to find infinitely many solutions $(u_1,s_1)$ of (\ref{eq04}) and therefore, in this way, infinitely many rational $D(q)$-quintuples. 

We found eight points $Q_i \in E(\QQ(u)), i\in \lbrace 1, \dots, 8 \rbrace,$ details in Table \ref{table:1}, each of them determining a $D(q(u))$-quintuple, such that the polynomial $P_{Q_i}(u)$ arising from the $D(q(u))$-quintuple is of degree three or four.

Define the curves
\[
E^{(i)}_{q}\colon \: P_{Q_i}(u)=qs^2,
\] 
for a fixed squarefree $q\in \mathbb{Z}$ and $i\in \lbrace 1 \dots 8\rbrace.$ If $q=1$ we write $E^{(i)}$ instead of $E^{(i)}_{1}.$ Each $E^{(i)}_{q}$ is a quadratic twist by $q$ of the curve $E^{(i)}.$

We want to find rational points on the curves $E^{(i)}_{q}.$ Let us look at a concrete example when $i=6,$ for which we have $P_{Q_6}(u)=4u^4 - 20u^3 + 13u^2 + 12u$. For each $q$, the curve $E^{(6)}_{q}$ has a rational point since $P_{Q_6}(u)$ has a rational zero $u=0.$ It follows that each $E^{(6)}_{q}$, as it is a curve of genus one, is birationally equivalent to an elliptic curve over $\QQ.$  

We want to classify squarefree $q\in \mathbb{Z}$ for which the rank of $E^{(6)}_{q}(\QQ)$ is positive. For such $q,$ equation (\ref{eq04}) has infinitely many rational solutions.

Let $E/\QQ$ be an elliptic curve. The \textit{root number} $W(E)$ is defined as the product of the local root numbers $W_p(E)\in \lbrace \pm 1 \rbrace:$
$$W(E)=\prod_{p\leq \infty} W_p(E),$$

\noindent where $p$ is a finite or infinite place of $\QQ.$ The local factors have the property that $W_p(E)=1,$ for all but finitely many $p.$ The definition of the local root number and their properties are explained in detail in e.g. \cite{Rohrlich2}. Rohrlich \cite{Rohrlich} provides an explicit formula for $W_p(E)$ when $p$ is not equal to $2$ or $3$ in terms of reduction types of $E.$ The remaining cases when $p=2$ or $p=3$ were covered by Halberstad \cite{Halb}. Rizzo \cite{Rizzo} gave a complete overview in English while removing some minimality conditions from the tables in \cite{Halb}.

The Birch and Swinnerton-Dyer conjecture implies the following

\begin{conjecture}[The Parity Conjecture]
Let $E/\QQ$ be an elliptic curve, then $\displaystyle (-1)^{\text{rank}E(\QQ)}=W(E).$
\end{conjecture}

An immediate consequence of this conjecture is that the rank of $E(\QQ)$ is positive whenever $W(E)=-1,$ in which case we have infinitely many rational points on $E.$

Assume the Parity Conjecture holds for all twists of the curves $E^{(i)}, i\in \lbrace 1, \dots, 8\rbrace.$ Using Desjardins \cite{Desj}, we obtain results for squarefree $q \mod N_i$ in the form of the following theorem.

\begin{theorem}\label{thm:02}
The functions $q \mapsto W\left(E^{(i)}_{q}\right)$ and $q \mapsto W\left(E^{(i)}_{-q}\right)$ are periodic on squarefree $q\in \mathbb{N}$ with period $N_i.$ Consequently, assuming the Parity Conjecture, the functions $q \mapsto \text{Rank}\left(E^{(i)}_{q}\right) \hspace{-5pt}\mod 2$ and $q \mapsto \text{Rank}\left(E^{(i)}_{-q}\right) \hspace{-5pt}\mod 2$ are periodic on squarefree $q\in \mathbb{N}$ with period $N_i.$
\end{theorem}

Each point $Q_i$ in Table \ref{table:1} leads to a different polynomial $P_{Q_i}(u).$ The period $N_i$ will depend on the periods of the local root numbers $W_p\left(E^{(i)}_{q}\right)$ with respect to $q,$ for each fixed prime $p$ dividing the conductor of $E^{(i)}.$ We will explicitly calculate $N_i$ for each curve $E^{(i)}$ using \cite{Desj}, with the help of tables in \cite{Halb} and  \cite{Rizzo}.

Combining results from all curves $E^{(i)}$, with the assumption of the Parity Conjecture, we prove the following theorem.

\begin{theorem}\label{thm:03}
Assuming the Parity Conjecture the following holds:

\begin{enumerate}[label=(\alph*)]
\item For each squarefree $q\in \mathbb{N}$ in at least $295026$ residue classes $\hspace{-5pt}\mod 394680$ there exist infinitely many rational $D(q)$-quintuples. 

\item For each squarefree $q\in -\mathbb{N}$ in at least $295435$ residue classes $\hspace{-5pt}\mod 394680$ there exist infinitely many rational $D(q)$-quintuples.
\end{enumerate}
\end{theorem} 
\begin{remark}
There are $296010$ residue classes $\hspace{-5pt}\mod 394680$ which contain squarefree integers. Theorem \ref{thm:03} shows that we cover more than $99.5\%$ of classes $\hspace{-5pt}\mod 394680.$ We conjecture that Theorem \ref{thm:03}  holds for all squarefree $q\in \mathbb{Z},$ that is for all $q\in \Q,$ but are unable to prove it using this method. 
\end{remark}

\section{Initial grunt work, constructing quintuples}\label{sec:02}

Following Dujella \cite{Euler}, we wish to find $D(q)$-quintuples of the form $\lbrace A, B, C, D, x^2\rbrace$ with\\ $q=\alpha\cdot x^2.$ Dujella started from the $D(q)$-pair $\lbrace B,C \rbrace$, with $BC+\alpha x^2=k^2.$ The numbers $A=B+C-2k$ and $D=B+C+2k$ both extend the pair $\lbrace B,C \rbrace$ to a regular $D(q)$-triple. The quadruple $\lbrace A,B,C,D \rbrace$ is an almost rational $D(q)$-quadruple, missing the condition $AD+\alpha x^2=\square.$ To obtain a rational $D(q)$-quintuple $\lbrace A, B, C, D, x^2\rbrace$ we also need to satisfy that $Y\cdot x^2+\alpha x^2=(Y+\alpha)x^2=\square,$ for $Y=A,B,C$ and $D.$

\begin{proposition}
Let $\lbrace A, B, C, D, x^2 \rbrace$ be a rational $D(\alpha x^2)$-quintuple with the properties
\begin{equation}\label{eq05}
A+\alpha=a^2,B+\alpha=b^2,C+\alpha=c^2,D+\alpha=d^2,
\end{equation}
\begin{equation}\label{eq06}
 BC+\alpha x^2=k^2,\quad  A=B+C-2k,\quad  D=B+C+2k.
\end{equation}

If we denote $p=\frac{d+a}{2}, r=\frac{d-a}{2}$ then
$$b^2=p^2+r^2-x^2+\frac{(p^2-x^2)(r^2-x^2)}{p^2+r^2-c^2-x^2}.$$
\end{proposition}

\begin{proof}

Subtracting the two rightmost equations in (\ref{eq06}), we have $$4k=D-A=(D+\alpha)-(A+\alpha)=d^2-a^2=(d-a)(d+a)=2r\cdot 2p.$$ It is easy to see that
\begin{equation}\label{eq07}
k=pr,\quad a=p-r,\quad d=p+r.
\end{equation}

The second equation from (\ref{eq06}), using (\ref{eq05}) and (\ref{eq07}), gives us
\begin{equation}\label{eq08}
(a^2-\alpha)=(b^2-\alpha)+(c^2-\alpha)-2k\quad\xRightarrow{(\ref{eq07})}\quad b^2+c^2=p^2+r^2+\alpha
\end{equation}
The first equation in (\ref{eq06}) gives us
$$k^2=(b^2-\alpha)(c^2-\alpha)+\alpha x^2.$$
Substituting $k=pr$ and manipulating using (\ref{eq08}), we obtain
$$4b^2c^2=4\cdot(p^2r^2+\alpha(p^2+r^2)-\alpha x^2).$$ Using the previous equality and (\ref{eq08}), we have
$$(b^2-c^2)^2=(b^2+c^2)^2-4b^2c^2=(p^2+r^2+\alpha)^2-4\cdot(p^2r^2+\alpha(p^2+r^2)-\alpha x^2).$$
Some more manipulations lead to
$$4(p^2-x^2)(r^2-x^2)=(\alpha-(p^2-x^2+r^2-x^2))^2-(b^2-c^2)^2.$$

The right hand side of the last equation is a difference of squares. Denoting
\begin{equation}\label{eq09}
2v=\alpha-(p^2-x^2+r^2-x^2)-(b^2-c^2),
\end{equation}
we have
\begin{equation}\label{eq10}
\frac{2(p^2-x^2)(r^2-x^2)}{v}=\alpha-(p^2-x^2+r^2-x^2)+(b^2-c^2).
\end{equation}
Adding (\ref{eq09}) and (\ref{eq10}), then diving by two, leads to
\begin{equation}\label{eq11}
\alpha=v+\frac{(p^2-x^2)(r^2-x^2)}{v}+(p^2-x^2)+(r^2-x^2)=\frac{1}{v}(p^2-x^2+v)(r^2-x^2+v).
\end{equation}

Subtracting (\ref{eq09}) from (\ref{eq10}) and dividing by two gives us
\begin{equation}\label{eq12}
b^2-c^2=\frac{1}{v}((p^2-x^2)(r^2-x^2)-v^2).
\end{equation}

Eliminating $\alpha$ from (\ref{eq08}) and (\ref{eq11}) gives us
\begin{equation}\label{eq13}
b^2+c^2=p^2+r^2+\frac{1}{v}(p^2-x^2+v)(r^2-x^2+v).
\end{equation}
Lastly, adding (\ref{eq12}) and (\ref{eq13}), as well as subtracting (\ref{eq12}) from (\ref{eq13}) and dividing by two, we have
\begin{equation}\label{eq14}
b^2=p^2+r^2-x^2+\frac{1}{v}(p^2-x^2)(r^2-x^2),
\end{equation}
\begin{equation}\label{eq15}
c^2=p^2+r^2-x^2-v.
\end{equation}
Substituting $v$ into (\ref{eq14}) using (\ref{eq15}), we finish with
\[
b^2=p^2+r^2-x^2+\frac{(p^2-x^2)(r^2-x^2)}{p^2+r^2-c^2-x^2}.
\]

\end{proof}

The previous proposition can be partially reversed.

\begin{proposition}\label{prop:6}
Let $p,r,c,x,b\in \Q$ such that
$$b^2=p^2+r^2-x^2+\frac{(p^2-x^2)(r^2-x^2)}{p^2+r^2-c^2-x^2}.$$
Define
$$a=p-r,\quad d=p+r,\quad k=pr,\quad \alpha=\frac{(c^2-r^2)(c^2-p^2)}{c^2+x^2-p^2-r^2},$$
$$A=a^2-\alpha,\quad B=b^2-\alpha,\quad C=c^2-\alpha,\quad D=d^2-\alpha.$$

Then $\lbrace A, B, C, D, x^2 \rbrace$ is a $D(\alpha x^2)$-quintuple provided that
\begin{enumerate}[label=(\roman*)]
\item no two elements of the quintuple are equal or equal to zero,
\item $\alpha$ is not equal to zero,
\item $AD+\alpha x^2=\square.$
\end{enumerate} 
\end{proposition}
\begin{proof}
One can check by calculation that the numbers $AB+\alpha x^2, AC+\alpha x^2, BC+\alpha x^2, BD+\alpha x^2, CD+\alpha x^2, Ax^2+\alpha x^2, Bx^2+\alpha x^2, Cx^2+\alpha x^2, Dx^2+\alpha x^2$ are squares. This proves the proposition.
\end{proof}

We now focus on rationality, and handle degeneracy issues in the proof of Theorem \ref{thm:09}.
\section{Reducing the number of parameters}\label{sec:03}

To find rational $D(q)$-quintuples we need rationals solutions of the pair of equations

\[
b^2=p^2+r^2-x^2+\frac{(p^2-x^2)(r^2-x^2)}{c^2+x^2-p^2-r^2}=p^2+r^2+\alpha-c^2,
\]

\[
z^2=AD+\alpha x^2=(p^2-r^2)^2+\alpha(x^2-2(p^2+r^2)+\alpha),
\]
where $\alpha$ is defined as
\[
\alpha=\frac{(c^2-r^2)(c^2-p^2)}{c^2+x^2-p^2-r^2}.
\]

We notice that $\alpha, b^2$ and $z^2$ are equal to homogeneous rational functions in $p,r,c,x$ so we start by setting $r=1.$ After that, the expressions for $\alpha, b^2, z^2$ simplify to
\[
\alpha=\frac{(c^2-1)(c^2-p^2)}{c^2+x^2-p^2-1},
\]
\begin{equation}\label{eq16}
b^2=p^2+1+\alpha-c^2,
\end{equation}
\begin{equation}\label{eq17}
z^2=(p^2-1)^2+\alpha(x^2-2(p^2+1)+\alpha).
\end{equation}
We would like to specialize one of the parameters $c,p,x$ using the other two, since we do not know how to completely solve the pair of equations (\ref{eq16}), (\ref{eq17}). This specialization should keep the squarefree part of $\alpha$ as simple as possible. 

Define the surfaces $S_1$ and $S_2$ over $\Q$ by the following equations:
$$S_1\colon (c^2-1)(c^2-p^2)=0,\quad S_2\colon c^2+x^2-p^2-1=0,$$
which are the zero sets of the numerator and denominator of $\alpha$. The surface $S_1$ is the union of the four planes $c=\pm p$ and $c=\pm 1,$  while $S_2$ is a hyperboloid and their intersection is the union of the eight lines
\[
l_{1,2,3,4}\colon c=\pm 1, x=\pm p, \quad l_{5,6,7,8} \colon c=\pm p, x=\pm 1.
\] 
A heuristic derived from Section 3, Lemma 5 in \cite{DuKaPe} tells us we could find a good specialization if we find a low degree surface in variables $c,p,x$ which intersects both $S_1$ and $S_2$ at exactly the lines $l_i.$ The logical first choice are planes which contain two lines $l_i.$ Such planes have equations $x=\pm 1 \pm c \pm p$, so we set $x=c+p+1$ (changes of signs do not change anything relevant). In practice, the author came across this specialization when examining the family of $D(q(u))$-quintuples (\ref{eq26}), found by Dujella. After the specialization, the equations for $\alpha, b^2, z^2$ are
\begin{align*}
\alpha&=\frac{1}{2}(c-p)(c-1),\\ \numberthis \label{eq18}
b^2&=p^2+\frac{1-c}{2}p-\frac{1}{2}(c^2+c)+1,\\ \label{eq19} \numberthis
z^2&=p^4 + \frac{c-1}{2}p^3 - \frac{5c^2 + 3}{4}p^2 + \frac{c^2 - 1}{2}p +
    \frac{3c^4-5c^2+2c+4}{4}.
\end{align*}

We further reduce the number of parameters. Equation (\ref{eq18}) is a conic in variables $b,p$ with a rational point $(1,c).$ Using standard technique, rational points on (\ref{eq18}) can be parametrized by 
\begin{equation}\label{eq20}
p=\frac{u^2c+c/2+1/2-2u}{u^2-1},\quad b=\frac{u^2 - 3uc/2 - u/2 + 1}{u^2 - 1}, \quad u\in \QQ.
\end{equation}

Plugging the expression for $p$ from (\ref{eq20}) into (\ref{eq19}) makes the right hand side a polynomial of degree four in variable $c$ with coefficients in $\QQ(u).$ Multiplying both sides by $\displaystyle \left(\frac{(u^2-1)^2}{u^2-1/4}\right)^2$ leads to
\begin{equation}\label{eq21}
\mathcal{C}\colon z_1^2=z^2\cdot \left(\frac{(u^2-1)^2}{u^2-1/4}\right)^2=f_4(u)c^4+f_3(u)c^3+f_2(u)c^2+f_1(u)c+f_0(u),
\end{equation}
where $f_i(u)$ are rational functions in variable $u$ given by
\begin{align*}
f_4(u)&=u^4+u^2+7, \\
f_3(u)&=-3\cdot\frac{(u^3+3u-1)(2u^2+1)}{u^2-1/4},\\ \numberthis \label{eq22}
f_2(u)&=\frac{-16u^8 + 16u^7 + 242u^6 - 76u^5 + 199u^4 - 166u^3 + 47u^2 + 10u - 13}{8\cdot(u^2-1/4)^2},\\ 
f_1(u)&=3\cdot\frac{(u^3+3u^2+1/2)(u^4 - 11/2u^3 + 4u^2 - 3/2u + 1/2)}{(u^2-1/4)^2},\\ 
f_0(u)&=\frac{16u^8 + 16u^7 - 116u^6 + 40u^5 + 409u^4 - 308u^3 + 25u^2 - 20u + 19}{16(u^2-1/4)^2}.
\end{align*}
The curve $\mathcal{C}$, defined by (\ref{eq21}), is birationally equivalent to an elliptic curve over $\QQ(u)$ since it has a rational point $\displaystyle \left(c,z_1\right)=\left(1,\frac{4u(u-1)^2}{u^2-1/4}\right).$ It is birational to the curve in Weierstrass form
\begin{align*}
E\colon y^2=&x^3-27\cdot(256u^8 + 64u^7 - 1280u^6 + 1216u^5 + 3265u^4 - 2372u^3 + 310u^2 - 332u + 169)x \\ 
&+ 54\cdot(4096u^{12} + 1536u^{11} - 30624u^{10} - 18400u^9 + 74448u^8 + 125568u^7 -
    59313u^6\\ 
    & - 165978u^5 + 154773u^4 - 40360u^3 + 5187u^2 - 6474u + 2197).
\end{align*}

The points
\begin{align*}
S_1&=[48u^4 + 168u^3 - 9u^2 - 138u + 39 , -1944u^5 - 1944u^4 + 4374u^3 + 486u^2 - 972u],\\ 
S_2&=\Bigg[\frac{48u^6 + 588u^5 + 753u^4 - 1014u^3 + 24u^2 - 6u + 39}{u^2 + 2u + 1} ,\\\numberthis \label{eq23}
  &\frac{-5832u^8 - 25596u^7 - 6156u^6 + 48438u^5 - 8100u^4 + 324u^3 - 3240u^2 + 162u}{u^3 + 3u^2 + 3u + 1}\Bigg],\\ 
S_3&=\Bigg[\frac{48u^6 + 204u^5 - 855u^4 + 78u^3 + 2028u^2 - 1098u + 27}{u^2 - 6u + 9},\\
&\frac{-5832u^8 + 21060u^7 + 972u^6 - 94446u^5 + 102384u^4 + 34020u^3 -
    67392u^2 + 486u + 8748}{u^3 - 9u^2 + 27u - 27}\Bigg],\\
S_4&=[48u^4 + 492u^3 + 693u^2 - 84u - 69 , -5832u^5 - 19764u^4 - 15228u^3 + 3402u^2 + 2754u - 324],\\
S_5&=\Bigg[\frac{48u^6 + 12u^5 - 291u^4 + 66u^3 + 600u^2 + 66u - 69}{u^2 + 2u + 1} ,\\
 & \frac{-1080u^8 - 2484u^7 + 6480u^6 + 17550u^5 - 1512u^4 - 18468u^3 -
    3348u^2 + 2538u + 324}{u^3 + 3u^2 + 3u + 1}\Bigg]
\end{align*}
are independent points in the Mordell-Weil group $E(\QQ(u)).$ We used Magma \cite{Magma} to prove the independence of the points $S_i$ by checking that the elliptic regulator of these points is nonzero.\hypertarget{EnaPetorku}{}

Each rational point on $E$ determines a rational point $(c(u), z_1(u)) \in \mathcal{C}.$ From (\ref{eq20}) we obtain $p(u)$ and $b(u).$ We set $r(u)=1$ and $x(u)=c(u)+p(u)+1.$ According to Proposition \ref{prop:6}, each $c(u)$ defines a $D(\alpha(u) x(u)^2)$-quintuple $\displaystyle \lbrace A(u), B(u), C(u), D(u), x^2(u) \rbrace,$ unless a degeneracy occurs (two elements of the quintuple might be equal, some element or $\alpha(u)$ might be identically zero). The condition $A(u)D(u)+\alpha(u)x(u)^2\in (\QQ(u))^2$ is satisfied because the pair of rational functions $(c(u), z_1(u))$ satisfies equation (\ref{eq21}). 

For each point on $E$ of the form $\displaystyle \sum_{i=1}^5k_iS_i$ with $k_i\in\lbrace -6,  \dots,  6 \rbrace,$ assuming it defines a $D(\alpha(u)x(u)^2)$-quintuple, we calculate the degree of the polynomial $P(u)$ defined at (\ref{eq03}) using Magma. We did not obtain any polynomials of degree one or two. Every polynomial of degree four turned out to be reducible, some had a rational zero and some were products of two irreducible square polynomials in $\QQ[u].$ Polynomials of degree three and polynomials of degree four with a rational zero, such that the quintuple associated to them is a non degenerate $D(q(u))$-quintuple we call \textit{good} polynomials. 

Each good polynomial $P_0(u)$ defines an elliptic curve by the equation $y^2=P_0(u),$ and every quadratic twist $qy^2=P_0(u)$ of of such a curve is an elliptic curve over $\QQ$ as well (the twists of curves, where $P_0$ is of degree four, have a rational point with $y=0.$)

For any two different good polynomials which define elliptic curves with the same $j$-invariant, there is a $q_0\in \mathbb{Z}$ such that the quadratic $q_0$-twist of one curve is isomorphic over $\QQ$ to the other curve. This is true because the $j$-invariant of all our curves is not equal to $0$ or $1728$ \cite[Chapter X, Prop. 5.4]{Silv}. We only count one representative of each class of polynomials which define elliptic curves with the same $j$-invariant.

The following points on $E$ determine a $D(q(u))$-quintuple such that the polynomial $P(u)$ is good, and all of the associated polynomials $P_{Q_i}(u)$ define elliptic curves $E^{(i)}$ which have different $j$-invariants:
\begin{table}[h!]
\begin{center}
\begin{tabular}{ |c|c| } 
\hline
$Q_i\in E(\QQ(u))$ & $P_{Q_i}(u)$ \\
\hline
$-4S_1 - 2S_2 - 2S_3 + 3S_4 + 5S_5$ & $-1200u^3 + 1645u^2 - 410u - 35$ \\ 
$-4S_1 - S_2 - 2S_3 + 2S_4 + 4S_5$ & $-80u^4+148u^3-65u^2-12u+9$ \\ 
$-3S_1 - S_2 - 2S_3 + S_4 + 4S_5$ & $-28u^4-44u^3+157u^2-106u+21$ \\ 
$-3S_1 - S_2 - S_3 + 2S_4 + 3S_5$ & $112u^4-100u^3-93u^2+92u-11$ \\ 
$-2S_1 - S_2 - 2S_3 + 2S_4 + 4S_5$ & $300u^3-65u^2+16u+1$ \\ 
$-2S_1 - 2S_3 + S_4 + 3S_5$ & $4u^4 - 20u^3 + 13u^2 + 12u$ \\ 
$-S_1 -S_2 - S_3 + S_4 + 3S_5$ & $-40u^3-19u^2+38u+21$ \\ 
$-S_4 + S_5$ & $-144u^3+61u^2+94u-11$ \\ 
\hline
\end{tabular}
\end{center}
\caption{Points $Q_i \in E(\QQ(u))$ with polynomial $P_{Q_i}(u)$ defining $E^{(i)}\colon y^2=P_{Q_i}(u).$}
\label{table:1}
\end{table}

\section{Periodicity of root numbers of twists}

For $E/\QQ$ and $0\neq t\in \mathbb{Z},$ let $E_t$ denote its quadratic twist by $t.$ We also introduce some non-standard notation from Desjardins \cite{Desj}.

Given an integer $\beta\in \mathbb{Z}$ and a prime $p$, let $v_p(\beta)$ denote the greatest exponent of $p$ dividing $\beta.$ By $\beta_{(p)}$ we denote the number such that
\[
\beta=\beta_{(p)}\cdot p^{v_p(\beta)}.
\]
Similarly, if $\displaystyle d=\prod_i p_i^{e_i}$, we define $\beta_{(d)}$ to be the integer such that
\[
\beta=\beta_{(d)}\cdot \prod_i p_i^{v_{p_i}(\beta)}.
\]
Desjardins \cite[Theorem 1.2 b)]{Desj}, proved that the function 
\[
t\mapsto W(E_t)
\]
is periodic on squarefree $t$ of constant sign, assuming $j(E)\neq 0, 1728$. We calculate these periods, as well as give explicit formulae for $W\left(E^{(i)}_t\right)$ for the curves $E^{(i)}$ using \cite{Desj} and tables from \cite{Rizzo}. Note that none of the curves in our calculations have $j$-invariant equal to $0$ or $1728.$

\cite[Theorem 1.2 a)]{Desj} gives an explicit formula for the root number of a twist of an elliptic curve, whose $j$-invariant is not $0,1728$:

\[
W(E_t)=-W_2(E_t)\cdot W_3(E_t)\cdot \left(\frac{-1}{|t_{(6\Delta)}|}\right)\cdot \left(\prod_{p|\Delta_{(6)}}W_p(E_t)\right),
\]
where $\left(\frac{\cdot}{\cdot}\right)$ is the Jacobi symbol.

Each factor in the previous equation is periodic on squarefree $t$ of constant sign. This is a consequence of the properties of the Jacobi symbol, and \cite[Lemma 3.2]{Desj}, which states that the function $t\mapsto W_p(E_t)$ is periodic on squarefree $t,$ for every prime $p$. Moreover, the same lemma proves that for $p\geq 5,$ the period of $W_p(E_t)$ divides $p^2$, and for $p=2$ or $3,$ the period is $p^{\gamma_p},$ for a nonnegative integer $\gamma_p.$ For explicit curves $E$ we can calculate $\gamma_p$ using tables in \cite{Halb} or \cite{Rizzo}.

We now state and prove an expanded version of Theorem \ref{thm:02} describing the curve $E^{(6)}$. A similar version of the following theorem (with similar proofs) can be 
made for every curve in Table \ref{table:1}.
\begin{theorem}\label{thm:07}
The curve $E^{(6)}$ has Weierstrass form $y^2=x^3-24003x+1296702,$ conductor $C=30$ and its discriminant is $\Delta=2^{14}3^{18}5^2.$

a) The periods of the functions $W_2(E^{(6)}_t), W_3(E^{(6)}_t)$ and $ W_5(E^{(6)}_t)$ on squarefree $t$ are $8, 3$ and $5,$ respectively.
The period of the function $\left(\frac{-1}{t_{(30)}}\right)$ on positive squarefree $t$ is $24.$

b) If $t>0, t'<0,$ both $t$ and $t'$ are squarefree and $t \equiv t' \:\:\:\:(\mmod 120),$ then $\left(\frac{-1}{t_{(30)}}\right)=-\left(\frac{-1}{-t'_{(30)}}\right)$. Specially, $W(E^{(6)}_t)=-W(E^{(6)}_{t'}).$

c) The period of $W(E^{(6)}_t)=-W_2(E^{(6)}_t)W_3(E^{(6)}_t)W_5(E^{(6)}_t)\left(\frac{-1}{|t_{(30)}|}\right)$ on squarefree $t$ of constant sign is $120.$

d) Assuming the Parity Conjecture, if $q$ is positive, squarefree and in one of $47$ classes $\hspace{-5pt}\mod 120,$ then the rank of $E^{(6)}_q(\QQ)$ is positive. 

If $q$ is negative, squarefree and in one of $43$ classes $\hspace{-5pt}\mod 120,$ then the rank of $E^{(6)}_q(\QQ)$ is positive. 

\end{theorem}
\begin{proof}
a) Assume $t$ is positive. We first prove $\left(\frac{-1}{t_{(30)}}\right)=\left(\frac{-1}{t_{(6)}}\right).$ 

If $5\nmid t,$ then obviously $t_{(30)}=t_{(6)}$ so $\left(\frac{-1}{t_{(30)}}\right)=\left(\frac{-1}{t_{(6)}}\right).$
Assume $t=5t'$ where $5\nmid t'.$ Then
$\left(\frac{-1}{t_{(30)}}\right)=\left(\frac{-1}{t'_{(6)}}\right)=\left(\frac{-1}{t'_{(6)}}\right)\left(\frac{-1}{5}\right)=\left(\frac{-1}{(5t')_{(6)}}\right)=\left(\frac{-1}{t_{(6)}}\right).$

To calculate $\left(\frac{-1}{n}\right)=-1^{(n-1)/2},$ for an odd number $n,$ we only need to know $n \mmod 4.$ Therefore, to prove $\left(\frac{-1}{t_{(6)}}\right)$ is periodic with period $24$ (on squarefree $t$) we check several cases. 

If $t=6t'$ then $t_{(6)}=t'$ since $t$ is squarefree. $\left(\frac{-1}{t'}\right)$ has period $4$ so the total period is $24$.

Cases $t=3t', t=2t'$ and $t=t'$ where in each case $(t',6)=1$ are handled similarly.

24 is the smallest period since $1=\left(\frac{-1}{3_{(6)}}\right)\neq \left(\frac{-1}{(11)_{(6)}}\right)=-1$ and $1=\left(\frac{-1}{2_{(6)}}\right)\neq \left(\frac{-1}{(14)_{(6)}}\right)=-1.$

We can calculate $W_5(E_t)$ using \cite[Prop 3.1]{Desj}. In our case, if $5\nmid t,$ the reduction of $E_t$ at $5$ is of type $I_2,$ while if $5 | t,$ the reduction is of type $I_2^{\ast},$ calculated by Magma\cite{Magma}. We conclude
$$W_5(E_t)=\Bigg\lbrace\begin{array}{c}
\hspace{5pt} 1\hspace{6pt}, \quad 5\mid t \\  \left(\frac{t}{5}\right), \quad 5\nmid t\end{array}.$$
 
 When $p=2$ or $3,$ things get more complicated. According to \cite[Prop 3.1]{Desj}, or \cite[1.1]{Rizzo} we need to find the smallest vector $(a',b',c')$ with nonnegative entries such that
 $$(a',b',c')=(v_p(c_4),v_p(c_6),v_p(\Delta))+k(4,6,12),$$
 for $k\in \Z,$ where $c_4, c_6, \Delta$ are the usual quantities associated to the Weierstrass equation of an elliptic curve.
 
 For $p=3$ and $t=1$, we have $(a',b',c')=(4,6,18)-1(4,6,12)=(0,0,6).$ Twisting by $t\not{\hspace{-4pt}\equiv} \hspace{4pt}0$ \\
 $(\mmod 3)$ does not change $(a',b',c').$ Per \cite[Table II,row 3]{Rizzo} we have that:
 
$t\equiv 1 \:\:(\mmod 3)\Rightarrow (c_6)_{(3)}\equiv 2 \:\:(\mmod 3),$  so $W_3(E_t)=-1.$
 
$t\equiv 2 \:\:(\mmod 3)\Rightarrow (c_6)_{(3)}\equiv 1 \:\:(\mmod 3),$  so $W_3(E_t)=1.$
 
 Twisting by $t\equiv 3,6 \:\:(\mmod 9)$ (note that $t$ is squarefree so it cannot be $\equiv 0 \:\:(\mmod 9)$)  we get in both cases $(a',b',c')=(6,9,24)-1(4,6,12)=(2,3,12).$ Per \cite[Table II,row 13]{Rizzo}, we have $W_3(E_t)=-1.$ In short,
 $$W_3(E_t)=\Bigg\lbrace\begin{array}{c}
 -1,\quad t\equiv 0,1 \:\:\:\:(\mmod 3) \\  \hspace{8pt}1, \hspace{11pt} t\equiv 2 \quad \:\:\:\:(\mmod 3)\end{array}.$$
 
 For $p=2$ and $t=1$, we have $(a',b',c')=(4,6,14)-1(4,6,12)=(0,0,2).$ Twisting by an odd $t$ will not change $(a',b',c')$. Per \cite[Table III, rows 2 and 9]{Rizzo} $W_2(E_t)=1$ if and only if $(c_6)_{(2)} \equiv 3$ $(\mmod 8),$ which happens exactly when $t\equiv 1 \:\:(\mmod 8).$
 
 Twisting by $t\equiv 2,6 \:\:(\mmod 8)$ makes $(a',b',c')=(6,9,20)-1(4,6,12)=(2,3,8).$ Per \cite[Table III, row 17]{Rizzo} we have:
 
 $t\equiv 2 \:\:(\mmod 8)\Rightarrow (c_6)_{(2)}\equiv 3 \:\:(\mmod 4),$  so $W_2(E_t)=1.$
 
$t\equiv 6\:\:(\mmod 8)\Rightarrow (c_6)_{(2)}\equiv 1 \:\:(\mmod 4),$  so $W_2(E_t)=-1.$
For $p=2$ we have
\[
W_2(E_t)=\Bigg\lbrace\begin{array}{c}
 \hspace{8pt}1,\quad t\equiv 1,2 \qquad \:\:\:\:(\mmod 8) \\   -1,\quad  t\equiv 3,5,6,7 \:\:\:\:(\mmod 8)\end{array}.
 \]
This concludes the proof of part a).

b) $\displaystyle \left(\frac{-1}{t_{(30)}}\right)=-\left(\frac{-1}{-t'_{(30)}}\right) \Leftrightarrow \left(\frac{-1}{t_{(30)}}\right)\cdot\left(\frac{-1}{-t'_{(30)}}\right)=-1 \Leftrightarrow \left(\frac{-1}{-t_{(30)}t'_{(30)}}\right)=-1.$ 

$t_{(30)}$ and $t'_{(30)}$ are odd numbers and since $t \equiv t' \:\:(\mmod 120)$ and $t,t'$ are not divisible by $4,$ because they are squarefree, then it must be true that $t_{(30)} \equiv t'_{(30)} \:\:(\mmod 4),$ and $-t_{(30)}t'_{(30)} \equiv -\left(t'_{(30)}\right)^2 \:\:(\mmod 4).$

Now, $\left(\frac{-1}{-t_{(30)}t'_{(30)}}\right)=\left(\frac{-1}{4k-\left(t'_{(30)}\right)^2}\right)=-1,$ since $-1$ is not a quadratic residue modulo a number of the form $4k'-1.$ $4k-\left(t'_{(30)}\right)^2\equiv -1 \:\:(\mmod 4)$ because $\left(t'_{(30)}\right)^2$ is an odd square.

c) Follows easily from a) and b).

d) If $t \mmod 120$ is $6, 7, 9, 11, 14, 15, 22, 26, 30, 35, 39, 41, 43, 50, 51, 53, 54, 58, 59, 61, 65, 66, 67, 71,$ $73, 74, 75, 77, 81, 82, 85, 86, 89, 90, 93, 95, 97, 99, 103, 105, 109, 110, 111, 114, 117, 118$ or $119$ and $t$ is positive and squarefree, then $W(E_t)=-1.$ Assuming the Parity Conjecture, the rank of $E_t(\Q)$ is odd, therefore positive.
 
\noindent If $t \mmod 120$ is $1, 2, 3, 5, 10, 13, 17, 18, 19, 21, 23, 25, 27, 29, 31, 33, 34, 37, 38, 42, 45, 46, 47, 49, 55,57,$ $ 62, 63, 69, 70, 78, 79, 83, 87, 91, 94, 98, 101, 102, 106, 107, 113$ or $115$ and $t$ is negative and squarefree, then $W(E_t)=-1.$ Assuming the Parity Conjecture, the rank of $E_t(\Q)$ is odd, therefore positive.
\end{proof}

\begin{remark}
For every curve $E^{(i)},$ the period of $W(E^{(i)})$ is divisible by $8.$ We use this to prove $W(E_{t}^{(i)})=-W(E_{t'}^{(i)}),$ for each curve $E^{(i)}.$
Each curve $E^{(i)}$ has a version of Theorem \ref{thm:07} similar to the one stated. We list important facts in Tables \ref{table:2} and \ref{table:3}.
\end{remark}
\begin{table}[h!]
\begin{center}
\begin{tabular}{ |c|c|c|c|c| } 
 \hline
 $E^{(i)}$ & Weierstrass form & $C$ & $\Delta$ & $N_i=\text{period of }W(E^{(i)}_t)$\\ 
 \hline
 $E^{(1)}$ & $y^2=x^3-33210675x+6964980750$ & $2\cdot3\cdot5^2\cdot11$ & $2^{20}3^{18}5^811^4$  & $6600=2^3\cdot3\cdot5^2\cdot11$\\
 $E^{(2)}$ & $y^2=x^3-24651x + 1453194$ & $2^4\cdot3\cdot13$ & $2^{10}3^{20}13$  & $312=2^3\cdot3\cdot13$\\
 $E^{(3)}$ & $y^2=x^3-97227x+10789254$ & $2\cdot3\cdot5\cdot11$ & $2^{16}3^{16}5^211^2$  & $1320=2^3\cdot3\cdot5\cdot11$\\ 
 $E^{(4)}$ & $y^2=x^3-7155x+187650$ & $2^4\cdot5^2\cdot11$ & $-2^{10}3^{12}5^311^2$  & $88=2^3\cdot11$\\
 $E^{(5)}$ & $y^2=x^3+274725x+126596250$ & $2^4\cdot3\cdot5^2\cdot11^2$ & $-2^{10}3^{18}5^611^3$  & $264=2^3\cdot3\cdot11$\\ 
 $E^{(6)}$ & $y^2=x^3-24003x+1296702$ & $2\cdot3\cdot5$ & $2^{14}3^{18}5^2$  & $120=2^3\cdot3\cdot5$\\ 
 $E^{(7)}$ & $y^2=x^3-132867x+17106174$ & $2\cdot3\cdot5\cdot11$ & $2^{16}3^{14}5^411^2$  & $1320=2^3\cdot3\cdot5\cdot11$\\
 $E^{(8)}$ & $y^2=x^3-1196883x+46619118$ & $2\cdot3\cdot5\cdot23$ & $2^{18}3^{22}5^223^2$  & $2760=2^3\cdot3\cdot5\cdot23$\\ 
 \hline
\end{tabular}
\end{center}
\caption{ }
\label{table:2}
\end{table}
\begin{table}[h!]
\begin{center}
\begin{tabular}{ |c|c|c| } 
 \hline
 $E^{(i)}$ & $W(E^{(i)}_t)$ given by local $W_p$ & periods of $W_p(E_t^{(i)})$ and $\left(\frac{-1}{|t_{(6\Delta)}|}\right)$\\ 
 \hline
 $E^{(1)}$ & $-W_2\left(E^{(1)}_t\right)\cdot W_3\left(E^{(1)}_t\right)\cdot W_5\left(E^{(1)}_t\right)\cdot W_{11}\left(E^{(1)}_t\right)\cdot\left(\frac{-1}{|t_{(330)}|}\right)$ & $8,3,25,11,132$\\
 \hline
 $E^{(2)}$ & $-W_2\left(E^{(2)}_t\right)\cdot W_3\left(E^{(2)}_t\right)\cdot W_{13}\left(E^{(2)}_t\right)\cdot\left(\frac{-1}{|t_{(78)|}}\right)$ & $8,3,13,24$\\
 \hline
 $E^{(3)}$ & $-W_2\left(E^{(3)}_t\right)\cdot W_3\left(E^{(3)}_t\right)\cdot W_5\left(E^{(3)}_t\right)\cdot W_{11}\left(E^{(3)}_t\right)\cdot\left(\frac{-1}{|t_{(330)}|}\right)$ & $8,3,5,11,132$\\
 \hline 
 $E^{(4)}$ & $-W_2\left(E^{(4)}_t\right)\cdot W_3\left(E^{(4)}_t\right)\cdot W_5\left(E^{(4)}_t\right)\cdot W_{11}\left(E^{(4)}_t\right)\cdot\left(\frac{-1}{|t_{(330)}|}\right)$ & $8,3,1,11,132$\\
 \hline
 $E^{(5)}$ & $-W_2\left(E^{(5)}_t\right)\cdot W_3\left(E^{(5)}_t\right)\cdot W_5\left(E^{(5)}_t\right)\cdot W_{11}\left(E^{(5)}_t\right)\cdot\left(\frac{-1}{|t_{(330)}|}\right)$ & $8,3,1,1,132$\\ 
 \hline
 $E^{(6)}$ & $-W_2\left(E^{(6)}_t\right)\cdot W_3\left(E^{(6)}_t\right)\cdot W_5\left(E^{(6)}_t\right)\cdot\left(\frac{-1}{|t_{(30)}|}\right)$ & $8,3,5,12$\\ 
 \hline
 $E^{(7)}$ & $-W_2\left(E^{(7)}_t\right)\cdot W_3\left(E^{(7)}_t\right)\cdot W_5\left(E^{(7)}_t\right)\cdot W_{11}\left(E^{(7)}_t\right)\cdot\left(\frac{-1}{|t_{(330)}|}\right)$ & $8,3,5,11,132$\\
 \hline
 $E^{(8)}$ & $-W_2\left(E^{(8)}_t\right)\cdot W_3\left(E^{(8)}_t\right)\cdot W_5\left(E^{(8)}_t\right)\cdot W_{23}\left(E^{(8)}_t\right)\cdot\left(\frac{-1}{|t_{(690)}|}\right)$ & $8,3,5,23,552$\\ 
 \hline
\end{tabular}
\end{center}
\caption{ }
\label{table:3}
\end{table}

\begin{theorem}\label{thm:09}
Let $q$ be a squarefree integer such that the rank of $E^{(i)}_q(\Q)$ is positive, for at least one $i\in \lbrace 1,\dots, 8\rbrace$. Then there exist infinitely many rational $D(q)$-quintuples.
\end{theorem}
\begin{proof}
Assume $i=6$ so that the rank of $E^{(6)}_q(\Q)$ is positive. If $i$ is any other index, the proof is similar. The quintuple
\[
\Big( 9(u-1)^3(4u-1)(u+1),
\]
\[
(u^4-6u^3+5u+27/4)(u^4-6u^3+8u^2-3u+3/4),
\]
\[
u^8-16u^6+32u^5-39/2u^4+10u^3+6u^2-9u+9/16,
\]
\[
(4u^4-16u^3+14u^2+4u+3)(u^4-2u^3+5/2u^2+3/4),
\]
\[
9(u^2-2u-1/2)^2(u^2+1/2)^2 \Big)
\]
is a $D(q(u))$-quintuple for 
\[
q(u)=(4u^4-20u^3+13u^2+12u)\cdot\left(3(u-1)(u^2+1/2)(u^2-2u-1/2)\right)^2.
\]
Evaluating the elements and $q(u)$ at $u_1\in \QQ$ we obtain a rational $D(q(u_1))$-quintuple for all but finitely many exceptions $u_1.$ The possible exceptions are rationals $u_1$ such that $q(u_1)=0,$ or any element of the function quintuple evaluated at $u_1$ is equal to zero, or any two elements of the function quintuple evaluated at $u_1$ are equal. Such $u_1$ are roots of finitely many polynomials in one variable so the set of exceptions is finite.

Since the rank of $E_q^{(6)}(\QQ)$ is assumed positive, we know there exist infinitely many pairs of rationals $(y_1,u_1)$ that satisfy the equation
\begin{equation}\label{eq24}
y^2q=4u^4-20u^3+13u^2+12u.
\end{equation} For fixed $y_1$ and $q$, the previous equation has at most four different solutions in variable $u$, so there are infinitely many different $y_1$ (and in a similar manner, infinitely many different $u_1$) among the pairs $(y_1,u_1)$ which satisfy (\ref{eq24}).

For each such pair $(y_1,u_1)$, let $\displaystyle \eta=\frac{1}{y_1\cdot 3(u_1-1)(u_1^2+1/2)(u_1^2-2u_1-1/2)}.$ It holds that $q(u_1)\cdot \eta^2=q.$

Then the quintuple
\[
\Big( 9(u_1-1)^3(4u_1-1)(u_1+1)\eta,
\]
\[
(u_1^4-6u_1^3+5u_1+27/4)(u_1^4-6u_1^3+8u_1^2-3u_1+3/4)\eta,
\]
\begin{equation}\label{eq25}
(u_1^8-16u_1^6+32u_1^5-39/2u_1^4+10u_1^3+6u_1^2-9u_1+9/16)\eta,
\end{equation}
\[
(4u_1^4-16u_1^3+14u_1^2+4u_1+3)(u_1^4-2u_1^3+5/2u_1^2+3/4)\eta,
\]
\[9(u_1^2-2u_1-1/2)^2(u_1^2+1/2)^2\eta \Big)
\]
is a rational $D(q)$-quintuple for all but finitely many exceptions of pairs $(y_1,u_1).$ The last thing left to argue is that the collection of rational $D(q)$-quintuples just described is not finite. For each such quintuple $( A,B,C,D,E )$ we look at the \textit{square} quintuple $( A^2, B^2, C^2, D^2, E^2 ).$ 

If the described collection of rational $D(q)$-quintuples were finite, then the collection of associated square quintuples would also be finite. Elements of square quintuples are rational functions in variable $u_1.$ It is an easy exercise to show that only finitely many different $u_1$ occur if there are only finitely many square quintuples. Since this is false, so is the assumption that there are only finitely many rational $D(q)$-quintuples described by (\ref{eq25}).
\end{proof}

\begin{proof}[Proof of Theorem \ref{thm:02}]
Theorem \ref{thm:07} c) implies Theorem \ref{thm:02} for the curve $E^{(6)}.$ The proofs for the other curves $E^{(i)}, i\neq 6$ are similar and omitted. The periods $N_i$ are listed in Table \ref{table:2}.
\end{proof}

\begin{proof}[Proof of Theorem \ref{thm:03}]
The least common denominator of the periods $N_i$ from Theorem \ref{thm:02} is $394680.$ The proof for negative $q$ is conducted in the same way as the proof for positive $q,$ so we assume $q$ is a squarefree positive integer. Theorem \ref{thm:07} d) implies that if $q \mmod 120$ is in one of forty seven residue classes the rank of $E_q^{(6)}(\QQ)$ is positive. Combining results for other curves $E^{(i)}$ we conclude that if $q$ is in one of $295026$ residue classes $\mmod 394680$ at least one $E^{(i)}_q(\QQ)$ has positive rank. Theorem \ref{thm:09} concludes our proof.
\end{proof}
For completeness, we list the $D(q(u))$-quintuples for all $E_i,  i \in \lbrace 1,\dots, 8\rbrace.$
\[
\Big( 900u^4+4320u^3-1161u^2-3438u+1404,
\]
\[
1600u^4-1600u^3+1100u^2-920u+396,
\]
\begin{equation}\label{eq26}
100u^4+1760u^3-1201u^2-542u+324,
\end{equation}
\[
2500u^4-4000u^3+959u^2+514u+36,
\]
\[
3600u^4-2880u^3-1584u^2+864u+324\Big)
\]
is a $D\left((-1200u^3 + 1645u^2 - 410u - 35)\cdot\left[6(10u^2-4u-3)\right]^2\right)$-quintuple,
\[
\Big( 378u^2 - 405u + 108,
\]
\[
32u^4 - 64u^3 + 122u^2 - 117u + 36,
\]
\[
32u^4 - 16u^3 + 80u^2 - 78u + 18,
\]
\[
128u^4 - 160u^3 + 26u^2 + 15u,
\]
\[
288u^4 - 288u^3 + 90u^2 - 9u\Big)
\]
is a $D\left((-80u^4+148u^3-65u^2-12u+9)\cdot\left[3(4u-1)\right]^2\right)$-quintuple,
\[
\Big( 352u^4 - 244u^3 - 129u^2 + 122u - 20,
\]
\[
4u^6 + 16u^5 + 48u^4 + 48u^3 - 164u^2 + 104u - 20,
\]
\[
4u^6 - 24u^5 + 112u^4 - 120u^3 + 47u^2 - 14u + 4,
\]
\[
16u^6 - 16u^5 - 32u^4 + 100u^3 - 105u^2 + 58u - 12,
\]
\[
36u^6 - 96u^5 + 112u^4 - 88u^3 + 48u^2 - 16u + 4\Big)
\]
is a $D\left((-28u^4-44u^3+157u^2-106u+21)\cdot\left[2(3u^2-u+1)\cdot(u-1)\right]^2\right)$-quintuple,
\[
\Big( -54u^2 + 171u - 90,
\]
\[
32u^4 - 96u^3 - 6u^2 + 127u - 30,
\]
\[
32u^4 + 144u^3 - 24u^2 - 26u - 18,
\]
\[
128u^4 + 96u^3 - 6u^2 + 31u - 6,
\]
\[
288u^4 + 576u^3 - 54u^2 - 117u - 18\Big)
\]
is a $D\left((112u^4-100u^3-93u^2+92u-11)\cdot\left[3\cdot(4u+1)\right]^2\right)$-quintuple,
\[
\Big( 450u^4 - 1665u^3 + 2052u^2 - 909u + 72,
\]
\[
50u^4 - 545u^3 + 1092u^2 - 317u + 44,
\]
\[
800u^4 - 350u^3 + 30u^2 - 158u + 2,
\]
\[
1250u^4 - 125u^3 + 192u^2 - 41u + 20,
\]
\[
4050u^4 - 405u^3 - 648u^2 - 81u\Big)
\]
is a $D\left((300u^3-65u^2+16u+1)\cdot\left[9\cdot(5u+1)\cdot(u-1)\right]^2\right)$-quintuple,
\[
\Big( 576u^5 - 1296u^4 + 288u^3 + 1152u^2 - 864u + 144,
\]
\[
16u^8 - 192u^7 + 704u^6 - 736u^5 - 72u^4 - 80u^3 + 624u^2 - 264u + 81,
\]
\[
16u^8 - 256u^6 + 512u^5 - 312u^4 + 160u^3 + 96u^2 - 144u + 9,
\]
\[
64u^8 - 384u^7 + 896u^6 - 1024u^5 + 528u^4 - 128u^3 + 288u^2 + 48u + 36,
\]
\[
144u^8 - 576u^7 + 576u^6 - 288u^5 + 504u^4 + 144u^3 + 144u^2 + 72u + 9\Big)
\]
is a $D\left((4u^4-20u^3+13u^2+12u)\left[12\cdot(2u^2+1)\cdot(2u^2-4u-1)\cdot(u-1)\right]^2\right)$-quintuple,
\[
\Big( 25u^2 + 30u + 20,
\]
\[
4u^2 + 24u + 20,
\]
\[
9u^2 - 2u - 4,
\]
\[
u^2 + 14u + 12,
\]
\[
16u^2 - 4\Big)
\]
is a $D\left((-40u^3-19u^2+38u+21)\cdot 2^2\right)$-quintuple, and
\[
\Big( 324u^4 + 423u^2 - 198u + 180,
\]
\[
64u^4 + 320u^3 - 52u^2 - 248u + 60,
\]
\[
100u^4 - 256u^3 + 239u^2 + 106u + 36,
\]
\[
4u^4 + 128u^3 - 49u^2 - 86u + 12,
\]
\[
144u^4 - 576u^3 + 432u^2 + 288u + 36\Big)
\]
is a $D\left((-144u^3+61u^2+94u-11)\cdot\left[6\cdot(2u^2-4u-1)\right]^2\right)$-quintuple.

\section*{Acknowledgements}

The author was supported by the Croatian Science Foundation under the project no.~1313. 

The author is grateful to Andrej Dujella for sharing his unpublished results, as well as motivation for this paper and to Julie Desjardins for helpful comments.

The author is very grateful to his mentor Matija Kazalicki for guidance through this paper.

\end{document}